\documentclass[11pt]{amsart}

\usepackage{amssymb}
\usepackage{enumitem}
\usepackage{booktabs, multirow}
\usepackage{array}
\usepackage{url}
\usepackage[utf8]{inputenc}
\usepackage{amsmath, hyperref}

\makeatletter
\newcommand{\labeltarget}[1]{\Hy@raisedlink{\hypertarget{#1}{}}}
\makeatother

\newcommand{\PreserveBackslash}[1]{\let\temp=\\#1\let\\=\temp}
\newcolumntype{C}[1]{>{\PreserveBackslash\centering}p{#1}}

\newtheorem{thm}{Theorem}[section]

\newtheorem{lem}[thm]{Lemma}
\newtheorem{prop}[thm]{Proposition}

\theoremstyle{definition}

\newcommand{\C}{\mathbb{C}}%Complex plane
\newcommand{\D}{\mathbb{D}}%Unit disk in Complex plane
\newcommand{\norm}[1]{\left\Vert#1\right\Vert}%norm
\newcommand{\abs}[1]{\left\vert#1\right\vert}%absolute value/modulus
\newcommand*\CLOSED[1]{\overline{#1}}%closure
\newcommand{\PAIR}[1]{\langle#1\rangle}%pairing

\newcommand*{\HOLO}{\mathcal H}%Holomorphic fucntions

\newcommand*\COP{\mathcal K}  %Compact operators
\newcommand*\BOP{\mathcal L}  %Bounded operators
\DeclareMathOperator{\ARG}{arg}%argument of a complex number/angle in polar form of complex number
\DeclareMathOperator{\CONV}{conv}%dimension
\newcommand*\BLOCH{\mathcal{B}}

\newcommand*\APPRE{\mathcal{AE}} %Approximate evaluation maps

\let\originalleft\left
\let\originalright\right
\renewcommand{\left}{\mathopen{}\mathclose\bgroup\originalleft}
\renewcommand{\right}{\aftergroup\egroup\originalright}

\allowdisplaybreaks

\begin{document}

\title[Essential norm and integration\ldots]{Essential norm and integration of a family of weighted composition operators}

\author[ ]{David Norrbo}

\email{d.norrbo@reading.ac.uk}
\address{Department of Mathematics and Statistics, School of Mathematical and Physical Sciences, University of Reading, Whiteknights, PO Box 220, Reading RG6 6AX, UK.}

\thanks{}

\keywords{Generalized Hilbert matrix operator, Hilbert matrix operator, Essential norm, Volterra operator, Weighted Bergman spaces, Weighted composition operators}

\subjclass{47B91, 47B38, 47G10, 30H20, 47A30} % Operators on complex function spaces, Operators on function spaces (general), Integral operators, Bergman spaces, Norms
%\subjclass{47B01, 47B33}   Operators on Banach spaces, Linear composition operators
%\date{\today}

\begin{abstract}
We study the interchange of essential norm and integration of certain families of weighted composition operators acting on the standard weighted Bergman spaces \(A^p_\alpha\), where \(p>1\) and \(\alpha\geq 0\). To be more precise, we give a sufficient condition for \( \norm{\int u_tC_{\phi_t}\, dt}_e = \int \norm{ u_tC_{\phi_t}}_e \, dt  \) to hold in terms of geometric properties of \(u_t\) and \(\phi_t\). We also provide some necessary conditions for the equality to hold and calculate the essential norm of some integral operators such as some Volterra operators.
\end{abstract}

\maketitle

\section{Introduction}

For \(t\in]0,1[ \), let \(u_t\colon \D\to \C\) and \(\phi_t\colon \D\to \D\) be holomorphic functions on the unit disk \(\D\) in the complex plane  \(\C\). The main purpose of this article is to examine the mean of composition operators \(\int_0^1 S_t \, dt\) acting on the weighted Bergman spaces  \( A^p_{\alpha}\), where \(S_t : = u_t C_{\phi_t}\colon  A^p_\alpha \to A^p_\alpha \colon f \mapsto u_t (f\circ\phi_t) \). In particular, we are interested in when the phenomenon
\begin{equation}\label{eq:equalityOfEssentialNorm}
 \norm{\int_0^1  S_t   \, dt}_e     =  \int_0^1 \norm{S_t}_e  \, dt
\end{equation}
occurs. Such an equality could prove to be useful in determining the essential norm of integral operators. In the special case when \(S_t = T_1\) for \(t\) in an interval \(I\subsetneq ]0,1[\) and \(S_t = T_2\) for \(t\in ]0,1[\setminus I\), where \(T_1\) and \(T_2\) are bounded linear operators, we always have \(\leq\) in \eqref{eq:equalityOfEssentialNorm}. In this case, an equality would mean that there is no cancellation between the noncompact behavior of the operators. In the general case, it might happen that \(S_t\) is compact for all \(t\), but \(\int_0^1 S_t \, dt\) is not compact, in which case the equality in \eqref{eq:equalityOfEssentialNorm} would be \(>\); an example is provided right after Proposition \ref{prop:basicIneqEsentialNorm}.

In \cite[p.~29]{Lindstrom-2022}, \eqref{eq:equalityOfEssentialNorm} was proved to hold for some weighted composition operators, where the composition symbol was given by \(\phi_t(z) := \frac{t}{1-(1-t)z}, z\in \D \). One of the restricting factors in calculating the essential norm for weighted composition operators or the mean of a family of them, is the fact that the methods used to obtain a precise upper bound demand the composition operators to be univalent, see e.g. \cite[Proof of Theorem 4.2 and Lemma 8.2]{Lindstrom-2022}. Another way to approach the problem is to use estimates via Carleson measures, see for example \cite{Cuckovic-2007}. However, evaluating the estimates are not always possible. As a consequence, Carleson measure techniques are not suited to approach operators of the form \(\int_0^1  S_t   \, dt\), except in degenerate cases, for example, if \(S_t=S_{1/2}\) for all \(t\in]0,1[\), in which case the equality in \eqref{eq:equalityOfEssentialNorm}  is trivial.

Assume that \(\phi_t(\D)\) touches \(\partial \D\) at only one point, which is a fixed point for all \(t\). If in addition \(\{ \phi_t : t\in]0,1[\}\) and \(\{u_t : t\in]0,1[\}\) satisfy certain geometric \emph{easy-to-check} conditions (see Section \ref{sec:EssNormIntOp}), we prove that \eqref{eq:equalityOfEssentialNorm} holds.

Many integral operators have a representation, where the conditions are satisfied, see for example Section \ref{sec:examples}. We also show that if \(\phi_t(\D)\) touches \(\partial \D\) at two points, the situation is more complicated, and the equality above does not necessarily hold even for very simple symbols. In order to obtain these results, we use the fact that the essential norm can be represented as a supremum of limits of weakly null sequences, which is a nontrivial fact when \(p\neq 2\), and by the geometric conditions we are able to find a dominating sequence among all weakly null sequences. This method does not utilize upper bounds for the essential norm in the usual sense, and as a consequence, \(\phi_t\) is not required to be univalent.

Given a Banach space \(X\), the Banach algebra of bounded linear operators on \(X\) is denoted by \(\BOP(X)\) and the ideal of compact operators is denoted by \(\COP(X)\). The set of weakly null sequences consisting of elements of a set \(M\subset X\) is denoted by \(W_0(M)\). The closed unit ball of \(X\) is denoted by \(B_X\). The open disk in the complex plane with center \(z\) and radius \(r\) is denoted by \(B(z,r)\).
For the set of analytic functions on the unit disk \(\D\subset \C\) we write \(\HOLO(\D)\), and the normalized Lebesgue area measure is given by \(dA(x+iy)=dx\, dy/\pi\). Let \(\alpha\geq 0\) and \(p>1\). The weighted Bergman space \(A^p_\alpha\) is the reflexive Banach space defined as
\[
\Big\{ f\in \HOLO(\D) :  \norm{f}^p := \int_{\D} \abs{f(z)}^p  dA_\alpha(z)  <\infty \Big\},
\]
where \(dA_\alpha(z) := (1+\alpha) (1-\abs{z}^2)^{\alpha} \, dA(z)\). For more information, we refer the reader to \cite{Zhu-2007}.

Given a Banach space, the essential norm of an operator \(S\in\BOP(X)\) is defined as
\[
\norm{S}_e := \inf_{K\in\COP(X)} \norm{S-K}.
\]
For a given \(1<p<\infty\), a Banach space \(X\) is said to have property \((M_p)\) if \(  \COP(X\oplus_p X) \) is an M-ideal in \( \BOP(X\oplus_p X) \) \cite[p.~497]{Werner-1992}. For all sections except Section \ref{sec:FurtherResults}, we will consider the quantity
\begin{equation}\label{eq:essNormRepresentation}
\sup_{(f_n)\in W_0(\partial B_X)}  \limsup_n  \norm{S_t f_n},
\end{equation}
which coincides with the essential norm when \(X\) has property \((M_p)\), see for example \cite[p.~499]{Werner-1992}. Justification for \( A^p_{\alpha} \) having \((M_p)\) can be found in \cite[p.~163 and Corollary 3.6]{Kalton-1995}, see also \cite[Section 3]{Lindstrom-2022} for details.

The paper is organized as follows: In Section \ref{sec:EssNormIntOp}, we define the term \emph{admissible family}, which characterize the families \(\{ u_t C_{\phi_t} : t\in]0,1[ \}\) whose symbols \(u_t\) and \(\phi_t\) have suitable geometric properties. The main result is Theorem \ref{thm:essentialNormEquality}, where \eqref{eq:equalityOfEssentialNorm} is proved. Some additional interesting equalities are found in Theorem \ref{thm:DominatingAE}, and in Proposition \ref{prop:GoodProp} we prove that our test functions  \((f_{c,\xi})_c ,\ \xi\in\partial \D\)  serve as maximizers among all weakly null sequences with unit norm. In Section \ref{sec:NecessaryConditions}, we state some necessary conditions for \eqref{eq:equalityOfEssentialNorm} to hold. By a slight modification of the conditions characterizing the admissible families, we obtain a large class of families  \(\{ u_t C_{\phi_t} : t\in]0,1[ \}\)  for which  \eqref{eq:equalityOfEssentialNorm} fails. In Section \ref{sec:examples}, we use Theorems \ref{thm:essentialNormEquality} and \ref{thm:DominatingAE} to compute the essential norm for some Volterra type operators and generalized Hilbert matrix operators. We also provide a simple example where \(\phi_t\) is not univalent. Finally, in Section  \ref{sec:FurtherResults}, we give a slight improvement of an inequality in Kalton's work \cite[Theorem 2.4]{Kalton-1993}, by using the essential norm in the upper bound instead of the norm. This yields a new characterization for \(\COP(X)\) being an M-ideal in \(\BOP(X)\), when \(X\) is separable.

\section{Essential norm of integral operators}\label{sec:EssNormIntOp}

Let \(\{ \phi_t\colon \D\to \D : t\in]0,1[ \}\) be a family of self-maps of the disk and \(\{u_t\in \HOLO(\D) : t\in]0,1[\}\subset A^p_{\alpha}\). We say that the family \(\{u_t C_{\phi_t} : t\in]0,1[\}\) have the same direction if there exists a \(\xi\in\partial\D \), which is called the direction, such that for all \(t\in]0,1[\) it holds that

\begin{enumerate}[label=(\alph*)]
\item \(\phi_t\in \HOLO(  \D  )\) and \(\phi_t'\in C(\CLOSED{\D} \cap B(\xi,\epsilon) )\) for some \(\epsilon = \epsilon(t) > 0\),   \label{C:phiAnalyticAtCriticalPoint}
\item \(  \phi_t(\xi) = \xi  \),   \label{C:xiFixPointForPhi}
\item \( \CLOSED{  \phi_t(\D\setminus  B(\xi,\epsilon)  ) }\cap \partial \D = \emptyset\) for every \(\epsilon>0\),   \label{C:PhiDoesNotTouchBoundaryExceptAtXi}
\item For every \(\epsilon>0\) the set \(\CLOSED{\phi(B(\xi,\epsilon)\cap \D)}\) contains a horocycle touching \(\xi\),    \label{C:NeigborhoodContainsHorocylce}
\item \(u_t\in C(\CLOSED{\D}\cap B(\xi,\epsilon) )\) for some \(\epsilon= \epsilon(t)>0\).   \label{C:utContinuousAtCriticalPoint}
\end{enumerate}
By \ref{C:phiAnalyticAtCriticalPoint}, the Julia-Carathéodory theorem (see e.g. \cite[Theorem 2.44]{Cowen-1995}) and the formula on the bottom of page 48 in \cite{Cowen-1995}, we have \(\phi_t'(\xi)> 0\) implying \(\phi_t'\) is univalent in some neighborhood of \(\xi\). The conditions \ref{C:xiFixPointForPhi} and \ref{C:PhiDoesNotTouchBoundaryExceptAtXi} ensures \(\phi_t(\D)\) touches the boundary exactly at one point, \(\xi\), which is a fixed point. Conditions \ref{C:NeigborhoodContainsHorocylce} and \ref{C:utContinuousAtCriticalPoint} guarantees that for a suitable \(\xi\in\partial \D\), the family \((f_{c,\xi})\) is a dominating approximate evaluation map; these concepts are defined in the next paragraph. The meaning of dominating is the statement of Proposition \ref{prop:GoodProp}, which is also the main reason for condition \ref{C:phiAnalyticAtCriticalPoint}. The meaning of \(\phi_t'\in C(\CLOSED{\D} \cap B(\xi,\epsilon) )\)  is that \(\phi_t'\) can be continuously extended to \( \CLOSED{\D} \cap B(\xi,\epsilon)  \).

The sequence \( (f_{c_n,\xi}) \), where \(c_n<(2+\alpha)/p, \lim_n c_n = (2+\alpha)/p\),  \(f_{c_n,\xi} = \hat{f}_{c_n,\xi} / \norm{\smash{ \hat{f}_{c_n,\xi}}}\) and \(\hat{f}_{c_n,\xi} = (\xi-z)^{-c_n}\), will play a crucial role in this article. It is a special case of  \emph{approximate evaluation maps}, which are defined to be sequences \( (f_n)\subset \partial B_{A^p_{\alpha}} \) for which there exists a \(\xi\in\partial \D\) such that \(\lim_n \sup_{z\in\D\setminus B(\xi,\epsilon)} \abs{f_n(z)} = 0\) for every \(\epsilon>0\). For a given \(\xi\in\partial \D\), the family of such sequences is denoted by \(\APPRE(\xi)\) and \(\APPRE :=  \bigcup_{\xi\in\partial\D}\APPRE(\xi)\). Since \(W_0(\partial B_{A^p_\alpha})\) consists of sequences \((f_n)\subset \partial B_{A^p_\alpha} \) tending to zero uniformly on compact subsets of the unit disk \(\D \), \(\APPRE\subset W_0(\partial B_{A^p_{\alpha}})\). Concerning notations, we write \(f_{c,\xi}\) and \(\lim_c\) instead of  \(f_{c_n,\xi}\) and \(\lim_n\), when these functions are involved. We abbreviate  \(f_{c,1}\) by  \(f_c\).

We say a family of weighted composition operators  \(\{S_t : t\in ]0,1[\}\) is \hypertarget{admissible}{\emph{admissible}} if the members have the same direction, say \(\xi\in\partial\D\), and condition \eqref{eq:Wxi} holds, which states
\begin{equation}\labeltarget{eq:W}\tag{$W_{\xi}$}\label{eq:Wxi}
\begin{split}
 & \int_0^1 \norm{S_t}_{\BOP(A^p_\alpha)}  \, dt<\infty \quad \text{ and } \quad  \\
&\text{ for every }t_0\in]0,1[:\begin{cases} \lim_{t\to t_0}  \sup_{z\in\D} \abs{ \phi_t(z) -  \phi_{t_0}(z) }=  0, \\
\text{for every } \epsilon'>0  \
\lim_{t\to t_0}  \int_{\D\setminus B(\xi,\epsilon')} \abs{u_t(z) - u_{t_0}(z) }^p \, dA_{\alpha}(z) =  0,  \\
\text{for some } \epsilon'>0   \ \sup_{ \abs{t-t_0}<\epsilon' }  \sup_{z\in B(\xi,\epsilon')\cap \D} \abs{ u_t(z) }<\infty, \\
\text{for some } \epsilon'>0 \  \inf_{ \abs{t-t_0}<\epsilon' } \inf_{z\in  B(\xi,\epsilon')\cap \D} \abs{\phi_t'(z)} >0 .
%\text{ for some } \epsilon'>0 \ \{\phi_t' : \abs{t-t_0}<\epsilon' \} \text{ is equicontinuous on } B(\xi,\epsilon')\cap \D.
\end{cases}
\end{split}
\end{equation}
The main purpose of \eqref{eq:Wxi} is to ensure  \(\int_0^1 S_t \, dt\) is a well-defined bounded operator on \(A^p_{\alpha}\), see Proposition \ref{prop:IntStWellDefined}.

An admissible family \(\{S_t : t\in ]0,1[\}\) satisfies the following: By the formula on the bottom of page 48 in \cite{Cowen-1995}, the function 
\[
z \mapsto   \frac{   1-\abs{z }^2  }{   1-\abs{\phi_t(z)}^2  } ,
\]
is bounded and \(\limsup_{z\to \xi}\) exists and coincides with \( \phi_t'(\xi)^{-1} \). Moreover, using the Julia-Carathéodory theorem, we have for any \(t\in]0,1[\) 
\begin{equation}\label{eq:nonTangentialLimitForQuotient}
z\mapsto \frac{   \xi-z  }{   \xi-\phi_t(z)  }\in H^\infty \quad  \text{ and the nontangential limit }  \quad \angle \lim_{z\to 1}   \frac{   \xi-z  }{   \xi-\phi_t(z)  }   = \phi_t(\xi)^{-1}.
\end{equation}
To see that
\[
z\mapsto \frac{   \xi-z  }{   \xi-\phi_t(z)  }\in H^\infty
\]
is true, let \(I_{z,\xi}\subset \CLOSED{\D}\) be the line from \(z\in \D\) to \(\xi\). For \(z\in\D\cap B(\xi,\epsilon)\), where \(\epsilon>0\) is smaller than the one given by condition \ref{C:phiAnalyticAtCriticalPoint} and such that \(\abs{\phi_t'(z) - \phi_t'(\xi)}< \phi_t'(\xi)/2\), we have
\[
\abs{   \xi - \phi_t(z) } = \abs{  \int_{ I_{z,\xi}  } \phi_t'(w) \, dw } \geq  \abs{\xi-z}\phi_t'(\xi)/2.
\]

Before we proceed, we remark that the admissible families could also contain weighted composition operators \(S_t\) with \( \CLOSED{ \phi_t(\D) }\cap \partial \D = \emptyset \). In this case, the operator is compact and will not affect the results given that \(\phi_t'\) is interpreted as the angular derivative, which in this case has infinite value. If the family \(\{S_t\colon t\in]0,1[\} \) contains such functions, for these \(t_0\), we can ignore all local (with respect to \(t\)) conditions except \(\lim_{t\to t_0}  \sup_{z\in\D} \abs{ \phi_t(z) -  \phi_{t_0}(z) }=  0\). 

We begin by proving \(\int_0^1 S_t \, dt\) is a well-defined bounded linear operator if   \(\{S_t : t\in ]0,1[\}\) is \hyperlink{admissible}{admissible}.

\begin{prop}\label{prop:IntStWellDefined}
Assume \(\{S_t : t\in ]0,1[\}\) is an \hyperlink{admissible}{admissible}
 family. Then \( \int_0^1 S_t \, dt \) is a bounded linear operator \(A^p_{\alpha}\to A^p_{\alpha}\) and
\[
 \norm{\int_0^1  S_t   \, dt}_{\BOP(A^p_{\alpha})}    \leq  \int_0^1 \norm{S_t}_{\BOP(A^p_{\alpha})}  \, dt.
\]
\end{prop}
\begin{proof}
It is sufficient to assume \(\xi=1\). We will prove that for every \(f\in A^p_{\alpha}\) the map \( t\mapsto S_tf \) is continuous. An application of the theory stated on pages 295--298 in \cite{Katznelson-2004} would then yield the statements. To this end, let \(t_0,\epsilon\in]0,1[, \  f\in A^p_{\alpha}\) and choose \(\epsilon'\in ]0,  (1/2) \min\{t_0,1-t_0\}[\) small enough to grant \(E:=\sup_{\abs{t-t_0}<\epsilon'}  \sup_{z\in B(1,\epsilon')\cap\D} \abs{ u_t(z)  } < \infty\),  
\[
C:=\sup_{\abs{t-t_0}<\epsilon'}  \sup_{z\in B(1,\epsilon')\cap\D}  \abs{\phi_t'(z)  }^{-2} \Bigg(  \frac{   1-\abs{z }^2  }{   1-\abs{\phi_t(z)}^2  }   \Bigg)^{\alpha}  < \infty 
\]
 and \(\phi_t'|_{B(1,\epsilon')\cap\D}\) is injective for every \(t\) satisfying \(\abs{t-t_0}<\epsilon'\). Put
\[
I_1(\epsilon') :=  \int_{B(1,\epsilon')\cap\D} \abs{ u_t C_{\phi_t} f(z)  -  u_{t_0} C_{\phi_{t_0}} f(z) }^p \, dA_{\alpha}(z) ,
\]
\[
I_2(\epsilon') :=  \int_{\D \setminus B(1,\epsilon') } \abs{ (u_t(z)-u_{t_0}(z)) C_{\phi_t} f(z) }^p \, dA_{\alpha}(z)  
\]
and
\[
I_3(\epsilon') :=  \int_{\D \setminus B(1,\epsilon') } \abs{   u_{t_0}(z)  (C_{\phi_t} f(z)  - C_{\phi_{t_0}} f(z) ) }^p \, dA_{\alpha}(z). 
\]

For \( \abs{t-t_0} < \epsilon'  \), we have
\begin{align*}
I_1(\epsilon') & \lesssim_p   \sup_{\abs{t-t_0}<\epsilon'} \int_{B(1,\epsilon')\cap\D} \abs{ u_t(z)  f(\phi_t(z))  }^p \, dA_{\alpha}(z)  \leq   E^p  \sup_{\abs{t-t_0}<\epsilon'} \int_{B(1,\epsilon')\cap\D} \abs{   f(\phi_t(z))  }^p \, dA_{\alpha}(z) \\
%& =2^p  E^p \int_{\phi_t( B(1,\epsilon')\cap\D )} \abs{   f(z)  }^p  \abs{ (\phi_t|_{B(1,\epsilon')\cap\D}^{-1})'(z)   }^2  \Bigg( \frac{ 1 - \abs{\phi_t|_{B(1,\epsilon')\cap\D}^{-1}(z)}^2 }{  1-\abs{z}^2 } \Bigg)^{\alpha} \, dA_{\alpha}(z) \\
& \leq  E^p C \sup_{\abs{t-t_0}<\epsilon'} \int_{\phi_t( B(1,\epsilon')\cap\D )} \abs{   f(z)  }^p  \, dA_{\alpha}(z).
\end{align*}
Now choose \(\epsilon''\in ]0,\epsilon'[  \) small enough to obtain  
\(
I_1(\epsilon'')  < \epsilon.
\) 

Next, we consider \(I_2\). Pick \(\epsilon'''\in ]0,\epsilon''[ \) small enough to ensure  
\[ 
\sup_{z\in\D\setminus B(\xi,\epsilon'')} \abs{ \phi_t(z) -  \phi_{t_0}(z) } < \frac{   1 -  \sup_{z\in\D\setminus B(1,\epsilon'')}  \abs{ \phi_{t_0}(z) }   }{2},
\]
whenever \(\abs{t-t_0}< \epsilon''' \), which yields
\begin{equation}\label{eq:phitUniformlyBounded}
\sup_{z\in\D\setminus B(\xi,\epsilon'')} \abs{ \phi_t(z) } < \frac{   1 + \sup_{z\in\D\setminus B(1,\epsilon'')}  \abs{ \phi_{t_0}(z) }   }{2} =: z_0 < 1 .
\end{equation}
For \(0<\delta<\epsilon'''\) small enough, \cite[Theorem 4.14]{Zhu-2007} and \hyperlink{eq:W}{$(W_1)$} yield that \(\abs{t-t_0}< \delta \) implies
\begin{align*}
I_2(\epsilon'') &  = \int_{\D \setminus B(1,\epsilon'') } \abs{ (u_t(z)-u_{t_0}(z)) C_{\phi_t} f(z) }^p \, dA_{\alpha}(z)  \\
 &  \leq  \frac{  \norm{f}_{A^p_\alpha}^p   }{ (1-\abs{z_0})^{2+\alpha}  } \int_{\D \setminus B(1,\epsilon'') } \abs{ u_t(z)-u_{t_0}(z)}^p  \, dA_{\alpha}(z)  <\epsilon.
\end{align*}
Finally, consider \(I_3\). From the definition of \(E\) and the choice of \(\delta\), it follows that \(\sup_{\abs{t-t_0}<\delta} \norm{u_t}_{A^p_{\alpha}} <\infty\). Due to \eqref{eq:phitUniformlyBounded}, \(f\) is continuous on \(M(\epsilon'',\delta):=\CLOSED{ \bigcup_{\abs{t-t_0}<\delta}   \phi_t(\D \setminus B(1,\epsilon''))} \). We can therefore choose \(\delta'\in]0,\delta[\) such that
\begin{align*}
 &\sup_{\abs{t-t_0}<\delta'} \norm{u_t}_{A^p_{\alpha}}^p  \sup_{ z\in \D \setminus B(1,\epsilon'') }  \abs{   (C_{\phi_t} f(z)  - C_{\phi_{t_0}} f(z) ) }^p \\
&\leq \sup_{\abs{t-t_0}<\delta'} \norm{u_t}_{A^p_{\alpha}}^p  \Big(  \sup_{ w\in M(\epsilon'',\delta) }  \abs{  f'(w) } \sup_{ z\in \D \setminus B(1,\epsilon'') } \abs{ \phi_t(z) - \phi_{t_0} (z) }   \Big)^p  < \epsilon.
\end{align*}
We have now obtained that \(\abs{t-t_0}<\delta'\) implies 
\(
I_3(\epsilon'') <\epsilon\), and hence,
\[
\sup_{\abs{t-t_0}<\delta'} \norm{S_t f - S_{t_0} f  }^p  \leq \sup_{\abs{t-t_0}<\delta'} (I_1(\epsilon'') + I_2(\epsilon'') +I_3(\epsilon''))  <3 \epsilon.
\]

\end{proof}

Concerning the essential norm, we have the following proposition.%With a stronger demand on the family \(\{S_t:t\in]0,1[\}\) concerning how closely related \(S_t\) and \(S_s\) are when \(\abs{t-s}\) is small, we have the following proposition. 
\begin{prop}\label{prop:basicIneqEsentialNorm}
Let \(X\) be a Banach space and assume \(\{T_t:t\in]0,1[\}\subset\BOP(X) \) is a family of operators such that  \(\int_0^1 \norm{T_t} \, dt<\infty\). If \(]0,1[\to \BOP(X):t\mapsto T_t\) is continuous, or \(t\mapsto T_tf\) is continuous for every \(f\in X\) and \(X\) has property \((M_p)\), then
\begin{equation}\label{eq:BasicEssentialNormIneq}
 \norm{\int_0^1  T_t   \, dt}_{e,\BOP(X)}    \leq  \int_0^1 \norm{T_t}_{e,\BOP(X)}  \, dt.
\end{equation}
\end{prop}
\begin{proof}
As in the proof of Proposition \ref{prop:IntStWellDefined}, the assumptions yield \(\int_0^1 T_t\, dt\) is a well-defined operator \cite[pp.~295--298]{Katznelson-2004}. Moreover, if \(t\mapsto T_t\) is continuous, so is \(t\mapsto \norm{T_t}_{e,\BOP(X)}\).

Let  \(\epsilon\in]0,1/2[\) and \(\pi\) be a bijection \(]\epsilon,1-\epsilon[\cap \mathbb Q \to \mathbb N := \{1,2,\ldots\}\). For \(q\in ]0,1[\cap \mathbb Q\) let \(K_q \colon  X \to Y\) be a compact operator satisfying 
\[
\norm{T_{q} - K_q}_{X\to Y} < \frac{1}{2^{\pi(q)}}  +   \norm{T_{q}}_{e,\BOP(X)} . 
\]
It follows that
\begin{equation}\label{eq:leftHandSum}
\begin{split}
\norm{  \frac{1}{n-1}\sum_{\substack{ k=1\\  k/n\in[\epsilon,1-\epsilon] } }^{n-1}   T_{k/n} -  \frac{1}{n-1}\sum_{\substack{ k=1\\  k/n\in[\epsilon,1-\epsilon] } }^{n-1}   K_{k/n}  } &\leq   \frac{1}{n-1}\sum_{\substack{ k=1\\  k/n\in[\epsilon,1-\epsilon] } }^{n-1}  \norm{  T_{k/n} - K_{k/n}  } \\
& \leq    \frac{1}{n-1} +  \frac{1}{n-1}\sum_{\substack{ k=1\\  k/n\in[\epsilon,1-\epsilon] } }^{n-1}  \norm{  T_{k/n} }_{e,\BOP(X)}. 
\end{split}
\end{equation}
Since \( t\mapsto \norm{  T_t }_{e,\BOP(X)}\) is continuous, the right-hand side tends to
\[
\int_0^1 \norm{  T_t }_{e,\BOP(X)} \, dt,
\]
as \(n\to \infty\). The left-hand side of \eqref{eq:leftHandSum} is bigger than
\[
\norm{  \frac{1}{n-1}\sum_{ \smash{   \substack{ k=1\\  k/n\in[\epsilon,1-\epsilon] } }  }^{n-1}   T_{k/n}  }_{e,\BOP(X)} ,
\]
hence, we can conclude
\[
\norm{  \int_\epsilon^{1-\epsilon} T_t \, dt  }_{e,\BOP(X)}   \leq  \int_\epsilon^{1-\epsilon} \norm{  T_t }_{e,\BOP(X)} \, dt.
\]
Letting \(\epsilon\to 0\) yields the statement.

In the case of \(X\) having property \((M_p)\), the inequality follows from formula \eqref{eq:essNormRepresentation}.

\end{proof}

Notice that there are compact operators \(K_t\) such that \(\int_0^1 K_t \, dt\) is non compact. In this instance we have
\[
\norm{ \int_0^1  K_t \, dt }_e > 0 = \int_0^1 \norm{K_t}_e \, dt.
\] 
An example is the standard integral representation of the Hilbert matrix on Bergman spaces, in which case \(S_t \colon f\mapsto  [z\mapsto f(t)/(1-tz)] \). The fact that \( \norm{S_t}\not\in L^1(]0,1[)\) can be proved directly using \(f(z) =  \Big( (1-z)^\frac{2}{p} \ln \frac{e}{1-z} \Big)^{-1} \), or indirectly, because if \( \norm{S_t}\in L^1(]0,1[)\), dominated convergence and \eqref{eq:essNormRepresentation} would give a contradiction; see \cite[Example 7.4]{Lindstrom-2022} or Section \ref{sec:examples} with \(\lambda=1\) for the value of \(\norm{ \int_0^1  S_t \, dt }_{e,\BOP(A^p_0)}\).

For the first statement in the lemma below and the remark following it, we refer to \cite[Section 4]{Thesis-2023}. The second statement in the lemma follows easily from the definition.

\begin{lem}\label{lem:BasicApproxEvalLemma}
Let \(g\in H^\infty\) and \(\xi\in\partial\D\). If the nontangential limit \(\angle \lim_{z\to \xi} \abs{g(z)}\) exists, then
\[
\lim_c  \norm{f_{c,\xi} g}_{A^p_{\alpha}} = \angle \lim_{z\to \xi} \abs{g(z)}.
\]
If \(g\in C(\CLOSED{\D}\cap B(\xi,\epsilon))\) for some \(\epsilon>0\), then for every \((f_n)\in \APPRE(\xi)\)
\[
\lim_n  \norm{f_n g}_{A^p_{\alpha}} =  \abs{g(\xi)}.
\]
\end{lem}

An interesting approximate evaluation map, which will not be used here, is \( h_{n,\xi} = \hat{h}_{n,\xi} / \smash{ \norm{\smash{ \hat{h}_{n,\xi}}}_{A^p_{\alpha}} }\), where \(\hat{h}_{n,\xi} =((\xi+z)/2)^n\). If \(g\) is a bounded function on \(\D\) for which the tangential limits exists, denoted \(t-\lim_{z\to \xi} \abs{g(z)}\) (the nontangential limits does not need to exist), then 
\[
\lim_n  \norm{h_{n,\xi} g}_{A^p_{\alpha}} = t-\lim_{z\to \xi} \abs{g(z)}.
\]
The approximate evaluation \((h_{n,\xi})\) is, therefore, in some way the opposite of \((f_{c,\xi})\).

Concerning the lemma below, notice that if \(\phi_t\) is continuous on \(\CLOSED{\D}\), then \(\phi_t^{-1}(\eta)\) is a finite set for any point \(\eta\in\partial\D\). 
\begin{lem} \label{lem:AEofIntNormSt}
Assume \(\{S_t:t\in]0,1[\}\) has direction \(\xi\in\partial\D\) except for condition \ref{C:xiFixPointForPhi} being replaced by \(\CLOSED{\phi_t(\D)}\cap \partial \D = \{\xi\}\) and \(\phi_t^{-1}(\xi)\) is finite, and for the rest of the conditions, \(B(\xi,\epsilon)\) is replaced by 
\(
\bigcup_{\zeta\in\phi_t^{-1}(\xi)} B(\zeta,\epsilon)
\). For \(t\in]0,1[\), we have
\[
 \lim_c \norm{S_t f_{c,\xi}}_{A^p_{\alpha}}^p \leq \sup_{ (f_n)\in \APPRE}\limsup_n \norm{S_t f_n}_{A^p_{\alpha}}^p  \leq     \sum_{\zeta\in \phi_t^{-1}(\xi)}  \frac{ \abs{u_t(\zeta)}^p }{ \abs{\phi_{t}'(\zeta)}^{2+\alpha}    },
\]
where the inequalities are equalities if \(\alpha=0\).

In particular, if \(\int_0^1\norm{S_t}_{\BOP(A^p_{\alpha})} \, dt < \infty\), then
\[
\lim_c  \int_0^1\norm{S_t f_c}_{A^p_{\alpha}}^p \, dt  \leq \sup_{ (f_n)\in \APPRE}\limsup_n \int_0^1\norm{S_t f_n}_{A^p_{\alpha}}^p \, dt \leq    \int_0^1 \sum_{\zeta\in \phi_t^{-1}(\xi)}  \frac{ \abs{u_t(\zeta)}^p }{ \abs{\phi_{t}'(\zeta)}^{2+\alpha}    } \, dt,
\]
where the inequalities are equalities if \(\alpha=0\).
\end{lem}

\begin{proof}
It is sufficient to consider \(\xi = 1\). Let \(t\in]0,1[\), \((f_n)\in \APPRE(1) \) and choose \(\epsilon>0\) small enough, to grant that for \(\zeta\in \phi_t^{-1}(1)\), \(u_t|_{B_{\zeta}}\) is continuous, \(\phi_t|_{B_{\zeta}}\) is invertible and has a continuous derivative, where \(B_{\zeta} := B(\zeta,\epsilon)\cap\CLOSED{\D}\). Now \((f_n)\in \APPRE(1) \) yields
\begin{align*}
&\limsup_n \norm{S_t f_n}_{A^p_{\alpha}}^p  =   \limsup_n    \sum_{\zeta\in \phi_t^{-1}(1)}  \int_{B_{\zeta}}  \abs{f_n(\phi_t(z))}^p  \abs{ u_t(z)}^p   \, dA_\alpha(z)  \\
&=   \sum_{\zeta\in \phi_t^{-1}(1)} \limsup_n   \int_{ \phi_t(B_{\zeta})}  \abs{f_n(z)}^p  \abs{ u_t\big(  {\phi_t}|_{B_{\zeta}}  \hspace{-4pt}^{-1}(z)  \big)}^p   \abs{  ({\phi_t}|_{B_{\zeta}} \hspace{-4pt}^{-1})'(z)  }^2   \Bigg( \frac{ 1- \abs{ \phi_t|_{B_{\zeta}} \hspace{-4pt}^{-1}(z) }^2  }{1-\abs{z}^2} \Bigg)^{\alpha} \, dA_\alpha(z) \\
%&\leq  \sum_{\zeta\in \phi_t^{-1}(1)}  \lim_n   \int_{\D}  \abs{f_n(z)}^p  \abs{ u_t\big(  {\phi_t}|_{B_{\zeta}}  \hspace{-4pt}^{-1}(z)  \big)}^p   \abs{  ({\phi_t}|_{B_{\zeta}} \hspace{-4pt}^{-1})'(z)  }^2  \, dA(z) \\
&\leq  \sum_{\zeta\in \phi_t^{-1}(1)}     \abs{ u_t\big(  {\phi_t}|_{B_{\zeta}}  \hspace{-4pt}^{-1}(1)  \big)}^p   \abs{  ({\phi_t}|_{B_{\zeta}} \hspace{-4pt}^{-1})'(1)  }^{2+\alpha} =  \sum_{\zeta\in \phi_t^{-1}(1)}  \frac{ \abs{u_t(\zeta)}^p }{ \abs{\phi_{t}'(\zeta)}^{2+\alpha}    }. 
\end{align*}
Since \( \phi_t(B_{\zeta}) \) contains a disk tangent to the unit circle at \(1\), the inequality is an equality if \((f_n) = (f_{c_n})\) and \(\alpha = 0 \).

Finally, take \((f_n)\in \APPRE(\eta) \) for some \(\eta\in\partial\D\setminus\{1\}\). Since  \(\CLOSED{\phi_t(\D)}\cap \partial \D\subset \{1\}\) there is an \(\epsilon'>0\) such that for \(n\) large enough \(f_n\) is bounded on \(\D\setminus B(\eta,\epsilon')\supset \phi_t(\D)\), and 

\[
 \int_{\D}  \abs{f_n(\phi_t(z))}^p  \abs{ u_t(z)}^p   \, dA_\alpha(z) \leq   \sup_{z\in\D\setminus B(\eta,\epsilon')} \abs{f_n(z)}   \norm{u_t}^p_{A^p_{\alpha}}  \stackrel{n\to \infty}{\longrightarrow} 0.
\]

The statements concerning the integrals follow directly from dominated convergence.
\end{proof}

\begin{lem}\label{lem:AEofIntSt}

If \(\{S_t:t\in]0,1[\}\) is an \hyperlink{admissible}{admissible} family with direction \(\xi\), then
\[
 \lim_c \norm{  \int_0^1 S_t f_{c,\xi} \, dt}_{A^p_{\alpha}}  =   \abs{ \int_0^1  \frac{      u_t(\xi)     }{   \phi_t'(\xi)^{(2+\alpha)/p}  } \, dt  }.
\]

\end{lem}
\begin{proof}
It is sufficient to assume \(\xi=1\). By the triangle inequality followed by Minkowski's inequality, we have 
\begin{align*}
\abs{  \norm{  \int_0^1 S_t f_c \, dt}_{A^p_{\alpha}}    -   \abs{ \int_0^1  \frac{      u_t(1)     }{   \phi_t'(1)^{(2+\alpha)/p}  } \, dt  }  }&\leq  \bigg( \int_{\D}\abs{  f_c(z)     \int_0^1  \abs{ \frac{   S_t f_c   }{  f_c  } - \frac{u_t(1)}{ \phi_t'(1)^{(2+\alpha)/p}  }  } \, dt  }^p \, dA_\alpha(z) \bigg)^{\frac{1}{p}}\\
 &\hspace{-3cm}\leq  \int_0^1  \bigg( \int_{\D}\abs{ f_c(z)      \bigg( u_t(z)  \bigg(  \frac{  1 - z   }{  1 - \phi_t(z)   }\bigg)^c  - \frac{u_t(1)}{ \phi'(1)^{(2+\alpha)/p}  }  \bigg)  }^p \, dA_\alpha(z) \bigg)^{\frac{1}{p}}  \, dt.  
\end{align*}
Using the fact that
\[
  \bigg( \int_{\D}\abs{ f_c (z)     \bigg( u_t(z)  \bigg(  \frac{  1 - z   }{  1 - \phi_t(z)   }\bigg)^c  - \frac{u_t(1)}{ \phi'(1)^{(2+\alpha)/p}  }  \bigg)  }^p \, dA_\alpha(z) \bigg)^{\frac{1}{p}}   \leq 2\norm{S_t}_{\BOP(A^p_{\alpha})} 
\]
the dominated convergence theorem and \eqref{eq:nonTangentialLimitForQuotient} yield
\[
\lim_c \int_0^1   \bigg( \int_{\D}\abs{  f_c      \bigg( u_t(z)  \bigg(  \frac{  1 - z   }{  1 - \phi_t(z)   }\bigg)^c  - \frac{u_t(1)}{ \phi'(1)^{(2+\alpha)/p}  }  \bigg)  }^p \, dA_\alpha(z) \bigg)^{\frac{1}{p}}   \, dt = 0.
\]
\end{proof}
Combining Lemmas \ref{lem:AEofIntNormSt} and \ref{lem:AEofIntSt}, we have the following theorem.

\begin{thm}\label{thm:DominatingAE}
Let \(\{S_t:t\in]0,1[\}\) be an \hyperlink{admissible}{admissible} family with direction \(\xi\in\partial \D\). If \(\ARG u_t(\xi)\) is constant for all \(t\in ]0,1[\), where \(\xi\in\partial\D\) is the direction of the family, then all of the following quantities are equal:
\[
  \sup_{ (f_n)\in \APPRE}\limsup_n \norm{  \int_0^1 S_t f_n \, dt}_{A^p_{\alpha}} , \quad \sup_{ (f_n)\in \APPRE}\limsup_n \int_0^1\norm{S_t f_n}_{A^p_{\alpha}} \, dt,
\]
\[
 \lim_c \norm{  \int_0^1 S_t f_{c,\xi} \, dt}_{A^p_{\alpha}}, \quad  \lim_c  \int_0^1\norm{S_t f_{c,\xi}}_{A^p_{\alpha}} \, dt,
\]
\[
 \int_0^1  \frac{      \abs{ u_t(\xi) }    }{   \phi_t'(\xi)^{(2+\alpha)/p}  } \, dt .
\]
\end{thm}

Analysing the proof, it is evident that 
\[
 \lim_c \norm{  \int_0^1 S_t f_c \, dt}_{A^p_{\alpha}}  = \lim_c  \int_0^1\norm{S_t f_c}_{A^p_{\alpha}} \, dt =  \int_0^1 \angle\lim_{z\to1}  \frac{   \abs{   u_t(z)  }   }{   \phi_t'(z)^{(2+\alpha)/p}  } \, dt, 
\]
holds without assuming that \(\phi_t\) and \(u_t\) are continuous in \(\CLOSED{\D}\cap B(1,\epsilon) \) for some \(\epsilon>0\). The Julia-Carathéodory theorem yields the angular derivative exists, given that the nontangential limits at \(1\) exist for \(\phi_t'\) and \(u_t\). We only need to assume boundedness of certain quantities in order to use Lemma \ref{lem:BasicApproxEvalLemma}. However, a stricter continuity condition is essential for the proof of Proposition \ref{prop:GoodProp}.

Theorem \ref{thm:DominatingAE} shows that the approximate evaluation \((f_c)\) dominates the class \(\APPRE\) given that the direction of the family \(\{S_t : t\in]0,1[\}\) is \(1\). The following result shows that \((f_c)\) dominates a larger class of sequences.
\begin{prop}\label{prop:GoodProp}
If \(\{S_t:t\in]0,1[\}\) is an \hyperlink{admissible}{admissible} family with direction \(\xi\in\partial\D\), then
\[
\sup_{(f_n)\in W_0(\partial B_{A^p_{\alpha}})}  \limsup_n  \norm{S_t f_n}_{A^p_{\alpha}}  \leq  \lim_c  \norm{S_t f_{c,\xi}}_{A^p_{\alpha}},\quad t\in]0,1[.
\]
\end{prop}
\begin{proof}
It is sufficient to consider \(\xi=1\). Let \(t\in]0,1[\), \((f_n)\in W_0(\partial B_{A^p_{\alpha}})\). Let \(\delta\in]0,\phi_t'(1)^{-\alpha}[\) and choose \(\epsilon>0\) small enough so that
\[
F\colon z\mapsto  \abs{ u_t\big(  {\phi_t}|_{\CLOSED{B}}  \hspace{-4pt}^{-1}(z)  \big)}^p   \abs{  ({\phi_t}|_{ \CLOSED{B}} \hspace{-4pt}^{-1})'(z)  }^2
\]
is continuous on the compact set \(\phi_t(\CLOSED{B}) \), \(\abs{F(z)-F(1)}<\delta\) for every \(z\in \phi_t(\CLOSED{B}) \) and 
\[
\sup_{z\in B} \bigg( \frac{ 1-\abs{z}^2  }{   1-\abs{\phi_t(z)}^2     } \Bigg)^\alpha  \leq \phi_t'(1)^{-\alpha} + \delta, 
\]
where \(B:=\D\cap B(1,\epsilon) \) and \(F(1) := \lim_{z\in\phi_t(B),{z\to 1} } F(z)\). 

Since   \(\CLOSED{\phi_t(\D)}\cap \partial \D = \{1\}\) and \(\phi_t^{-1}(1)=\{1\}\),
\begin{align*}
 \limsup_n \norm{S_t f_n}_{A^p_{\alpha}}^p &=   \limsup_n \int_{\D} \abs{ f_n(\phi_t(z))  u_t(z)}^p  \, dA_\alpha(z) =  \limsup_n \int_{B} \abs{ f_n(\phi_t(z))  u_t(z)}^p  \, dA_\alpha(z)\\
&\leq   (\phi_t'(1)^{-\alpha} + \delta)  \limsup_n   \int_{ \phi_t(B)}  \abs{f_n(z)}^p  \abs{ u_t\big(  {\phi_t}|_{B}  \hspace{-4pt}^{-1}(z)  \big)}^p   \abs{  ({\phi_t}|_{B} \hspace{-4pt}^{-1})'(z)  }^2  \, dA_\alpha(z)\\
&\leq (\phi_t'(1)^{-\alpha} + \delta) ( F(1) + \delta )    \limsup_n \int_{ \phi_t(B)}  \abs{f_n(z)}^p  \, dA_\alpha(z). 
\end{align*}
Similar calculations yield 
\[
 \limsup_n \norm{S_t f_n}_{A^p_{\alpha}}^p  \geq (\phi_t'(1)^{-\alpha} - \delta) (  F(1)  -\delta  )  \limsup_n \int_{ \phi_t(B)}  \abs{f_n(z)}^p  \, dA_\alpha(z),
\]
if \(\epsilon>0\) is small enough. Letting \(\delta\to 0\), we can conclude
\begin{equation}\label{ineqForBestAPPRE}
 \limsup_n \norm{S_t f_n}_{A^p_{\alpha}}^p  = \phi_t'(1)^{-\alpha}   F(1)  \limsup_n  \int_{ \phi_t(B)}  \abs{f_n(z)}^p  \, dA_\alpha(z) \leq \phi_t'(1)^{-\alpha}   F(1) .
\end{equation}
The inequality in \eqref{ineqForBestAPPRE} is an equality when \( (f_n) = (f_{c_n}) \) since \(\phi_t(B)\) contains all nontangential limits due to \ref{C:NeigborhoodContainsHorocylce}.
\end{proof}

We are now ready to present the main theorem.

\begin{thm}\label{thm:essentialNormEquality}
If \(\{S_t : t\in ]0,1[\}\) is an \hyperlink{admissible}{admissible} family with direction \(\xi\in\partial \D\) and \(\ARG u_t(\xi)\) is constant for all \(t\in ]0,1[\), then
\[
\sup_{ (f_n) \in W_0(\partial B_{A^p_{\alpha}})}     \limsup_n  \norm{ \int_0^1 S_t f_n  \,  dt   }_{A^p_{\alpha}}   =       \int_0^1    \sup_{ (f_n) \in W_0(\partial B_{A^p_{\alpha}})}       \limsup_n     \norm{S_t  f_n  }_{A^p_{\alpha}}  \,  dt    = \lim_c   \int_0^1 \norm{S_t  f_{c,\xi}   }_{A^p_{\alpha}}  \,  dt  
\]
Moreover,
\[
\norm{ \int_0^1 S_t  \,  dt   }_{e,\BOP(A^p_{\alpha})}  = \int_0^1 \norm{ S_t   }_{e,\BOP(A^p_{\alpha})}   \,  dt.  
\]
\end{thm}
\begin{proof}
Let \(\xi\in\partial\D\) be the direction of \(\{S_t : t\in ]0,1[\}\). An application of Minkowski's inequality followed by dominated convergence, Proposition \ref{prop:GoodProp} and Theorem \ref{thm:DominatingAE} yield
\begin{align*}
&\sup_{ (f_n) \in W_0(\partial B_{A^p_{\alpha}})}     \limsup_n  \norm{ \int_0^1 S_t f_n  \,  dt   }_{A^p_{\alpha}}  \leq        \int_0^1              \sup_{ (f_n) \in W_0(\partial B_{A^p_{\alpha}})}       \limsup_n   \norm{S_t  f_n  }_{A^p_{\alpha}}  \,  dt  \\
&= \int_0^1 \lim_c    \norm{S_t  f_{c,\xi}   }_{A^p_{\alpha}}  \,  dt  = \lim_c  \norm{ \int_0^1 S_t  f_{c,\xi}     \,  dt  }_{A^p_{\alpha}}   \leq  \sup_{ (f_n) \in W_0(\partial B_{A^p_{\alpha}})}     \limsup_n  \norm{ \int_0^1 S_t f_n  \,  dt   }_{A^p_{\alpha}}. 
\end{align*}

Concerning the essential norm, it is sufficient to notice that \(A^p_\alpha\) has the \((M_p)\) property, in which case \eqref{eq:essNormRepresentation} coincides with the essential norm. 
\end{proof}

\section{Remarks concerning the necessity of admissibility for \eqref{eq:equalityOfEssentialNorm} to hold}\label{sec:NecessaryConditions}
In Theorem \ref{thm:essentialNormEquality}, we obtained a sufficient conditions for \eqref{eq:equalityOfEssentialNorm} to hold. In this section we will mention some necessary conditions for the equality to hold. One situation where \eqref{eq:equalityOfEssentialNorm} fails is mentioned right after Proposition \ref{prop:basicIneqEsentialNorm}. Another straightforward condition is obtained from the following:

\begin{prop}\label{prop:simpleNecCond}
Assume \(\{S_t:t\in]0,1[\}\subset\BOP(A^p_{\alpha}) \) is a family of operators such that \(t\mapsto S_tf\) is continuous for every \(f\in A^p_{\alpha}\) and \(\int_0^1 \norm{S_t}_{\BOP(A^p_{\alpha})} \, dt<\infty\). If there exists no sequence \((f_n)\in W_0(\partial B_{A^p_\alpha})\) such that \(\lim_n\norm{S_t(f_n)} = \norm{S_t}_e\) for almost every \(t\in]0,1[\), then
\[
\norm{ \int_0^1 u_tC_{\phi_t}  \,  dt   }_{e,\BOP(A^p_{\alpha})} < \int_0^1 \norm{ u_tC_{\phi_t}   }_{e,\BOP(A^p_{\alpha})}   \,  dt.  
\]
\end{prop}
\begin{proof}
By Minkowski's inequality and the assumption, we have
\begin{align*}
\sup_{ (f_n) \in W_0(\partial B_{A^p_{\alpha}})}     \limsup_n  \norm{ \int_0^1 S_t f_n  \,  dt   }_{A^p_{\alpha}}  <  \int_0^1 \norm{ u_tC_{\phi_t}   }_{e,\BOP(A^p_{\alpha})}   \,  dt.  
\end{align*}
\end{proof}

Assume \(\{S_t : t\in]0,1[\}\) is an \hyperlink{admissible}{admissible} family except for the fact that both \(1\) and \(-1\) is a direction, meaning that among the conditions \ref{C:phiAnalyticAtCriticalPoint} -- \ref{C:utContinuousAtCriticalPoint}, \(B(\xi,\epsilon)\) is replaced by \( B(1,\epsilon) \cup B(-1,\epsilon) \) and \ref{C:xiFixPointForPhi} is replaced by \( \phi_t(1) = 1, \phi_t(-1) = -1 \). Furthermore, similar changes are done to \eqref{eq:Wxi}, to obtain a condition we denote by \hypertarget{Wt}{$(W_{-1,1})$}.  Examples of composition symbols satisfying the conditions can be constructed from Riemann maps \(\psi_t\) mapping \(\D\) onto the interior of a \(C^2\)-curve in \(\CLOSED{\D}\) such that \ref{C:phiAnalyticAtCriticalPoint}, \ref{C:PhiDoesNotTouchBoundaryExceptAtXi} and \ref{C:NeigborhoodContainsHorocylce} are satisfied. Notice that the \(C^2\) condition grants that \(\psi_t\) can be continuosly extended to the boundary since the curve is Dini-smooth \cite[p.~42]{Pommerenke-1992}. Moreover, by premapping the disk \(\D\) by a suitable automorphism \(\sigma_t\) of the disk, the points  \(-1\) and \(1\) from the extended domain \(\CLOSED{\D}\) can be mapped to the points on \(\partial \D\) that are mapped by \(\psi_t\) to \(-1\) and \(1\), making  \(\psi_t\circ \sigma_t\) a candidate for \(\phi_t\). Define \(g_{c,\theta} := \hat{g}_{c,\theta} /  \norm{ \hat{g}_{c,\theta} }\), where \(\hat{g}_{c,\theta}  :=  \theta (1+z)^{-c} + (1-\theta)(1-z)^{-c} \).

\begin{lem}\label{lem:EstimatesWhenTwoDirections}
Assume \(\{S_t : t\in]0,1[\}\) has the directions \(-1\) and \(1\) and satisfies \hyperlink{Wt}{$(W_{-1,1})$}. Then 
\[
\lim_c \norm{  \int_0^1 S_t  g_{c,\theta}  \, dt   }_{A^p_{\alpha}} =      \Bigg(   \frac{  \theta^p }{    \theta^p + (1-\theta)^p }\abs{  \int_0^1  \frac{ u_t(1) }{ \phi_t'(1)^\frac{2+\alpha}{p}}   \, dt    }^p   +   \frac{  (1-\theta)^p }{    \theta^p + (1-\theta)^p } \abs{  \int_0^1  \frac{ u_t(-1) }{ \phi_t'(-1)^\frac{2+\alpha}{p}}   \, dt    }^p     \Bigg)^{\frac{1}{p}}
\]
and 
\begin{equation}\label{eq:gcTwoDirections}
\lim_c \int_0^1  \norm{   S_t  g_{c,\theta}     }_{A^p_{\alpha}}  \, dt =    \int_0^1  \Bigg(   \frac{  \theta^p }{    \theta^p + (1-\theta)^p }\abs{   \frac{ u_t(1) }{ \phi_t'(1)^\frac{2+\alpha}{p}}   }^p     +   \frac{  (1-\theta)^p }{    \theta^p + (1-\theta)^p } \abs{    \frac{ u_t(-1) }{ \phi_t'(-1)^\frac{2+\alpha}{p}}   }^p    \Bigg)^{\frac{1}{p}} \, dt .
\end{equation}
\end{lem}
\begin{proof}
Details concerning \(g_{c,\theta}\) are given in \cite[Lemma 3.2]{Norrbo-2025}. We will prove that the right-hand side of the inequality below tends to zero as \( c\to (2+\alpha)/p\):
\[
\begin{split}
&\int_{\D} \abs{  \int_0^1  g_{c,\theta}(z) \Bigg( \frac{ (S_t  g_{c,\theta})(z) }{ g_{c,\theta}(z) }   \, dt   - \chi_{\Re>0}(z) \int_0^1  \frac{ u_t(1) }{ \phi'(1)^\frac{2+\alpha}{p}}   \, dt   -   \chi_{\Re<0}(z) \int_0^1  \frac{ u_t(-1) }{ \phi'(-1)^\frac{2+\alpha}{p}}   \, dt \Bigg)  }^p \, dA_\alpha(z)   \\
&\lesssim_p \int_{\D} \abs{  g_{c,\theta}(z) \chi_{\Re>0}(z) \Bigg(  \int_0^1   \frac{ (S_t  g_{c,\theta})(z) }{ g_{c,\theta}(z) }   \, dt   -   \int_0^1   \frac{ u_t(1) }{ \phi'(1)^\frac{2+\alpha}{p}}   \, dt \Bigg)   }^p \, dA_\alpha(z) \\
&\quad + \int_{\D} \abs{  g_{c,\theta}(z) \chi_{\Re<0}(z) \Bigg(  \int_0^1   \frac{ (S_t  g_{c,\theta})(z) }{ g_{c,\theta}(z) }   \, dt    -  \int_0^1   \frac{ u_t(-1) }{ \phi'(-1)^\frac{2+\alpha}{p}}   \, dt   \Bigg)  }^p \, dA_\alpha(z).
\end{split}
\]
Since similar methods are used for the two terms, we only prove that the first tends to zero as \( c\to (2+\alpha)/p\). By Minkowski's inequality
\begin{align*}
&\norm{ g_{c,\theta} \chi_{\Re>0} \Bigg(  \int_0^1   \frac{ S_t  g_{c,\theta} }{ g_{c,\theta} }   \, dt   -  \int_0^1   \frac{ u_t(1) }{ \phi'(1)^\frac{2+\alpha}{p}}   \, dt \Bigg)   }_{A^p_{\alpha}} \\
&\leq \int_0^1  \norm{ g_{c,\theta} \chi_{\Re>0} \Bigg(  \frac{ S_t  g_{c,\theta} }{ g_{c,\theta} }    -    \frac{ u_t(1) }{ \phi'(1)^\frac{2+\alpha}{p}} \Bigg)   }_{A^p_{\alpha}} \, dt.
\end{align*}
Since
\begin{align*}
\frac{ (S_t  g_{c,\theta})(z) }{ g_{c,\theta}(z) }& = u_t (z)   \frac{ (1-\theta) (1+z)^c }{  (1-\theta) (1+z)^c  + \theta (1-z)^c      }  \frac{   (1-z)^c   }{  (1-\phi_t(z))^c  }  \\
&\quad +  u_t (z)   \frac{ \theta (1-z)^c }{ (1-\theta) (1+z)^c  + \theta (1-z)^c     }  \frac{   (1+z)^c   }{  (1+\phi_t(z))^c  }
\end{align*}
the dominated convergence theorem yields
\begin{align*}
&\lim_c \int_0^1  \norm{ g_{c,\theta} \chi_{\Re>0} \Bigg(  \frac{ S_t  g_{c,\theta} }{ g_{c,\theta} }    -     \frac{ u_t(1) }{ \phi'(1)^\frac{2+\alpha}{p}} \Bigg)   }_{A^p_{\alpha}} \, dt \\
&=  \int_0^1  \left(   \frac{  \theta^p }{    \theta^p + (1-\theta)^p }   \right)^{\frac{1}{p}}  \abs{    \frac{ u_t(1) }{ \phi'(1)^\frac{2+\alpha}{p}}  -     \frac{ u_t(1) }{ \phi'(1)^\frac{2+\alpha}{p}} }  \, dt = 0.
\end{align*}

This gives us the expression for \(\lim_c \norm{  \int_0^1 S_t  g_{c,\theta}  \, dt   }_{A^p_{\alpha}} \). To obtain \eqref{eq:gcTwoDirections}, let \(s\in]0,1[\) and consider the family obtained by putting \(S_t =S_s\) for every \(t\). Repeat this for every \(s\) and integrate over \(s\in]0,1[\) and we are done.

\end{proof}

We now present a proposition that illustrates the strong demand on \(u_t\) and \(\phi_t\) if the family \(\{S_t : t\in]0,1[\}\) have two directions.

\begin{prop}\label{prop:EssentialNormEqual}
Assume \(\{S_t : t\in]0,1[\}\) has the directions \(-1\) and \(1\) and satisfies \hyperlink{Wt}{$(W_{-1,1})$}. Let \(\theta\in [0,1]\) and let
\[
\Theta_t(1) := \Bigg(    \frac{  \theta^p }{    \theta^p + (1-\theta)^p }   \Bigg)^{\frac{1}{p}}    \frac{ u_t(1) }{ \phi_t'(1)^\frac{2}{p}}    \quad \text{ and }\quad \Theta_t(-1) := \Bigg(    \frac{  (1-\theta)^p }{    \theta^p + (1-\theta)^p }   \Bigg)^{\frac{1}{p}}    \frac{ u_t(-1) }{ \phi_t'(-1)^{\frac{2}{p}} }.  
\]
Then
\begin{equation}\label{eq:EssNormEqual}
\lim_c \norm{  \int_0^1 S_t  g_{c,\theta}  \, dt   }_{A^p_{\alpha}}   = \lim_c \int_0^1  \norm{   S_t  g_{c,\theta}     }_{A^p_{\alpha}}  \, dt
\end{equation}
if and only if there is a set \(M\subset]0,1[\) of full measure such that for \(t_0\in M\),
\[
\begin{pmatrix}  \Theta_{t} (-1) \\ \Theta_{t}(1)  \end{pmatrix}  =  \lambda_t \begin{pmatrix}  \Theta_{t_0} (-1) \\ \Theta_{t_0}(1)  \end{pmatrix} \quad \text{ or } \quad \begin{pmatrix}  \Theta_{t_0} (-1) \\ \Theta_{t_0}(1)  \end{pmatrix} = \begin{pmatrix}  0 \\ 0  \end{pmatrix}
\]
for all \(t\in M\) and some \(\lambda_t\geq 0\). 
\end{prop}
\begin{proof}
Consider \( \Theta_t\in \ell^p(\{-1,1\}\to \C) \). By applying Minkowski's inequality \(n\)-times, we have
\[
\norm{\sum_{k=1}^n \begin{pmatrix}  \Theta_{\frac{k}{n+1}} (-1) \\ \Theta_{\frac{k}{n+1}}(1)  \end{pmatrix}  }_{\ell^p} \leq \sum_{k=1}^n \norm{ \begin{pmatrix}  \Theta_{\frac{k}{n+1}} (-1) \\ \Theta_{\frac{k}{n+1}}(1)  \end{pmatrix}  }_{\ell^p}.
\]
Dividing both sides by \(n\) and letting \(n\to \infty\), we obtain

\begin{align*}
& \Bigg(   \frac{  \theta^p }{    \theta^p + (1-\theta)^p } \abs{  \int_0^1   \frac{ u_t(1) }{ \phi_t'(1)^\frac{2}{p}}     \, dt    }^p   +   \frac{  (1-\theta)^p }{    \theta^p + (1-\theta)^p } \abs{  \int_0^1   \frac{ u_t(-1) }{ \phi_t'(-1)^\frac{2}{p}}     \, dt  }^p     \Bigg)^{\frac{1}{p}}  \\
& \lim_n    \norm{   \frac{1}{n}  \sum_{k=1}^n \begin{pmatrix}  \Theta_{\frac{k}{n+1}} (-1) \\ \Theta_{\frac{k}{n+1}}(1)  \end{pmatrix}  }_{\ell^p} \leq \lim_n   \frac{1}{n}   \sum_{k=1}^n \norm{   \begin{pmatrix}  \Theta_{\frac{k}{n+1}} (-1) \\ \Theta_{\frac{k}{n+1}}(1)  \end{pmatrix}  }_{\ell^p} \\
&= \int_0^1  \Bigg(   \frac{  \theta^p }{    \theta^p + (1-\theta)^p }\abs{   \frac{ u_t(1) }{ \phi_t'(1)^\frac{2}{p}}   }^p     +   \frac{  (1-\theta)^p }{    \theta^p + (1-\theta)^p } \abs{    \frac{ u_t(-1) }{ \phi_t'(-1)^\frac{2}{p}}   }^p    \Bigg)^{\frac{1}{p}} \, dt.
\end{align*}
 
Therefore, in view of Lemma \ref{lem:EstimatesWhenTwoDirections}, the condition characterizing when the equality in \eqref{eq:EssNormEqual} is an equality (and not a strict inequality) is obtained from Minkowski's inequality.

\end{proof}

The restriction imposed by Minkowski's \(\ell^p\)-inequality, is due to the family of weighted composition operators having more than one direction. For \eqref{eq:EssNormEqual} to hold, the behavior at one direction will put a restriction on the behavior of \(u_t\) and \(\phi_t\) at every other direction, whenever the family has more than one direction. If the family is \hyperlink{admissible}{admissible} with direction \(\xi\), the condition is simplified to demand that there is a set of full measure \(M\subset ]0,1[\) such that  \(\ARG u_t(\xi) \) is constant for \(t\in M\). 

It is evident from Proposition \ref{prop:EssentialNormEqual} that even if there is a weakly null sequence \((f_n)\) such that \(\norm{S_t f_n}_{A^p_{\alpha}} = \norm{S_t}_{e,\BOP(A^p_{\alpha})} \) for every \(t\), we cannot claim that  \eqref{eq:EssNormEqual} holds, cf. Proposition \ref{prop:simpleNecCond}. 

Although the automorphisms of the disk can't be used as composition symbols to generate a family having only two directions, a similar result to Proposition \ref{prop:EssentialNormEqual} has been obtained in \cite{Norrbo-2025} concerning the family generated by \(u_t(z) := (1-zt)^{-1}\) and \(\phi_t(z) := (z-t)/(1-tz)\) and \(t\in ]-1,1[\), which fails to satisfy the \emph{if and only if}-condition given in  Proposition \ref{prop:EssentialNormEqual}, see  \cite[Theorems 3.4 and 3.8]{Norrbo-2025}.

\section{Examples}\label{sec:examples}
In this section we will use the theory developed in Section \ref{sec:EssNormIntOp} to calculate the essential norm of some Volterra operators and the generalized Hilbert matrix operator \(H_{\lambda}, \lambda>0\). Concerning the boundedness of these operators, we refer to  \cite[Theorem 1]{Aleman-1997} and \cite[Theorem 1.1]{Jevtic-2017}. Notice that on \(A^p_\alpha\), \(H_\lambda, \lambda>0\) is bounded if and only if \(H_1\) is bounded.

We also provide a simple example of an \hyperlink{admissible}{admissible} family, where \(\phi_t\) is not univalent.

\subsection{The Volterra operator}
In \cite[Theorem 1]{Aleman-1997} the authors characterize boundedness of the Volterra operator. The Volterra operators that we consider are
\[
f\mapsto \Big[z\mapsto V_g(f) (z) := \int_0^z f(t) g'(t) \, dt \Big],
\]
where \( z\mapsto  (1-z)g'(z)\) is bounded, and continuous in a neighborhood of \(1\) in \(\CLOSED{\D}\). Let \(p>1\), \(\alpha\geq 0\) and put
\[
\phi_t(z) := \frac{z t}{1-(1-t)z}   \quad  \text{and} \quad u_t(z) := \tau\frac{\phi_t(z)}{t} (1-\phi_t(z))  g'(\phi_t(z)), \quad z\in \D,
\]
where \(\tau\in \partial \D\) is chosen such that \(  \angle\lim_{z\to 1}\tau(1-z)g'(z)\geq 0\). The substitution \(t\mapsto \phi_t(z)\) yields
\[
 V_g(f) (z) = \int_0^1 u_t C_{\phi_t} \, dt.
\]
Clearly \(u_t, \phi_t\in \HOLO(\D)\) and it is easy to see that \ref{C:phiAnalyticAtCriticalPoint}-- \ref{C:NeigborhoodContainsHorocylce} hold with \(\xi=1\). To ensure \ref{C:utContinuousAtCriticalPoint} is satisfied, we only consider functions \(g\in\BLOCH\) such that \( z\mapsto  (1-z)g'(z)\) is continuous in a neighborhood of \(1\) in \(\CLOSED{\D}\). Since \(  \phi_t'(z) = t / ( 1-(1-t)z )^2     \), we can conclude that the four conditions in \((W_1)\) are satisfied. 

It remains to prove that \(\int_0^1 \smash{ \norm{u_t C_{\phi_t}}_{\BOP(A^p_{\alpha})} } \, dt < \infty \). Let \(M := \sup_{z\in \D} \abs{1-z} \abs{ g'(z) } <\infty \). Next, we use the substitution \(z\mapsto \phi_t^{-1}(z) = z/(t+(1-t)z)\) (notice that \(\abs{\phi_t^{-1}(z)} \geq \abs{z}, z\in \D\)) to obtain% (see \cite[Proof of Theorem 1.2]{Karapetrovic-2018})
%\begin{align*}
% \norm{u_t C_{\phi_t} f}^p &\leq M^p \frac{ t^{2+\alpha-p} }{ (1-t)^{2+\alpha} } \int_{B\big(\frac{1}{2-t},\frac{1-t}{2-t} \big)} \abs{z}^{2p-2\alpha-4} \abs{f\Big(\frac{z-t}{1-t}\Big)}^p    dA_\alpha(z)  \\
%&\leq M^p \frac{ t^{2+\alpha-p} }{ (1-t)^{2+\alpha} } \int_{B\big(\frac{1}{2-t},\frac{1-t}{2-t} \big)} \abs{f\Big(\frac{z-t}{1-t}\Big)}^p    dA_\alpha(z) \\
%&= M^p \frac{ t^{2+\alpha-p} }{ (1-t)^{2+\alpha} } \int_{B\big(\frac{1-t}{2-t},\frac{1}{2-t} \big)} \abs{f(z)}^p \Bigg( \frac{    1-\abs{t+(1-t)z}^2     }{   1-\abs{z}^2 }\Bigg)^\alpha  dA_\alpha(z).
%\end{align*} 
%Justify quotient is largest on the boundary near 1.

\begin{equation}\label{eq:UpperBoundStart}
 \norm{u_t C_{\phi_t} f}_{A^p_{\alpha}}^p \leq M^p  t^{2-p} \int_{B\big(\frac{1-t}{2-t},\frac{1}{2-t} \big)} \abs{f(z)}^p \frac{  \abs{z}^{p} }{ \abs{t+(1-t)z}^4} \, dA_\alpha(z).
\end{equation} 

Since \(  (t+(1-t)z)^{-1} \) is analytic on \(B\big(\frac{1-t}{2-t},\frac{1}{2-t} \big)\), by the maximum modulus principle  
\begin{align*}
\inf_{ z\in B\big(\frac{1-t}{2-t},\frac{1}{2-t} \big)} \abs{t+(1-t)z}^2 & = \inf_{ z\in \partial B\big(\frac{1-t}{2-t},\frac{1}{2-t} \big)} \abs{t+(1-t)z}^2 =  \inf_{z\in\partial\D }  \abs{t+(1-t)\Big(z\frac{1}{2-t}+\frac{1-t}{2-t}\Big)}^2  \\
&=   \inf_{z\in\partial\D }  \abs{z\frac{1-t}{2-t}+\frac{1}{2-t}}^2  = \frac{t^2}{  (2-t)^2 }\geq \frac{ t^2 }{4}.
\end{align*} 
It is, therefore, clear that
\begin{equation}\label{eq:largetUpperBound}
\int_{\frac{1}{6}}^1  \norm{u_t C_{\phi_t} f}_{A^p_{\alpha}}  \leq  \norm {f}_{A^p_{\alpha}}  12^{\frac{4}{p}} M \int_{\frac{1}{6}}^1   t^{\frac{2}{p}-1}  \, dt < \infty.
\end{equation}
Furthermore, for \( t < \frac{1}{6} \) and \(\frac{3}{5}\leq\abs{z}<1 \) we have
\begin{equation}\label{eq:smalltLargezUpperbBound}
 \abs{t+(1-t)z}   \geq (1-t)\abs{z}-t \geq  \frac{1}{3}.
\end{equation}
 For \( t < \frac{1}{6} \) and \( \abs{z} < \frac{3}{5} \) we use the evaluation estimates for \(f\in A^p_\alpha\) to conclude that
\begin{equation}\label{eq:smalltSmallzUpperbBound}
 \int_{B(0,\frac{3}{5})\cap B\big(\frac{1-t}{2-t},\frac{1}{2-t} \big)}  \frac{  \abs{f(z)}^p \abs{z}^{p} }{ \abs{t+(1-t)z}^4} \, dA_\alpha(z) \lesssim_{p} \int_{B(0,\frac{3}{5})\cap B\big(\frac{1-t}{2-t},\frac{1}{2-t} \big)}  \frac{    \norm{f}_{A^p_{\alpha}}^p  \abs{z}^{p} }{ \abs{t+(1-t)z}^4} \, dA(z). 
\end{equation}
Furthermore, using the substitutions \( z\mapsto  z \frac{1}{2-t}   + \frac{ 1-t }{ 2- t }  \) and \( z\mapsto  \frac{z - (1-t)}{1 - z (1-t)}   \), we obtain
\begin{align*}
\int_{B\big(\frac{1-t}{2-t},\frac{1}{2-t} \big)}  \frac{  \abs{z}^{p} }{ \abs{t+(1-t)z}^4} & \, dA(z)  =   \bigg(\frac{1}{2-t}\bigg)^{p-2} \int_{\D}  \frac{  \abs{    z    +  1-t   }^4 }{ \abs{    z(1-t)+ 1     }^4}  \abs{    z    +  1-t   }^{p-4}   \, dA(z) \\
&= t^{p-2} \int_{\D}   \frac{   \abs{z}^p    }{   \abs{ 1-(1-t)z}^p  }   \, dA(z) \lesssim \begin{cases}
 t^{p-2},  &\text{if }  1<p<2;\\
 \ln\frac{e}{t},  &\text{if }  p=2;\\
1, &\text{if }   p>2;
\end{cases}
\end{align*} 
where the last estimate can be found in, for example, Lemma 3.10 in \cite{Zhu-2007}. Combining this with \eqref{eq:UpperBoundStart}, \eqref{eq:largetUpperBound},  \eqref{eq:smalltLargezUpperbBound} and \eqref{eq:smalltSmallzUpperbBound}, we have for \(f\in B_{A^p_{\alpha}}\)
\begin{align*}
\int_0^1  \norm{u_t C_{\phi_t} f}_{A^p_{\alpha}}  \,  dt &\lesssim_{p}  M  \Bigg( \int_{\frac{1}{6}}^1   t^{\frac{2}{p}-1}  \, dt  +    \int_0^ {\frac{1}{6}} t^{\frac{2}{p}-1}  \Big( \int_{ B\big(\frac{1-t}{2-t},\frac{1}{2-t} \big) \setminus B(0,\frac{3}{5}) }    \abs{z}^{p}  \abs{f(z)}^p \, dA(z)  \Big)^{\frac{1}{p}} \, dt   \\
&\quad + \int_0^ {\frac{1}{6}}   \Big(   \int_{B(0,\frac{3}{5})\cap B\big(\frac{1-t}{2-t},\frac{1}{2-t} \big)}  \frac{  \abs{z}^{p} }{ \abs{t+(1-t)z}^4} \, dA(z)  \Big)^{\frac{1}{p}}  \, dt \Bigg) \lesssim_p  M . 
\end{align*}

We can now conclude that \(\{u_t C_{\phi_t} :  t\in]0,1[\}\) is an \hyperlink{admissible}{admissible} family. By Theorems \ref{thm:essentialNormEquality} and \ref{thm:DominatingAE}, we obtain
\[
\norm{V_g}_{e,\BOP(A^p_{\alpha})} = \int_0^1  \frac{      u_t(1)     }{   \phi_t'(1)^{(2+\alpha)/p}  } \, dt =   \lim_{z\in\D, z\to 1} \abs{(1-z)g'(z)}   \int_0^1  t^{\frac{2+\alpha}{p}}  \, dt .
\]

\subsection{The generalized Hilbert matrix operator}

The generalized Hibert matrix operator is given by
\[
f\mapsto \Big[z\mapsto H_\lambda(f) (z) := \int_0^1\frac{ f(t) t^{\lambda-1} }{ 1-tz  } \, dt \Big]
\]
and bounded when  \(p>2+\alpha\geq 2\), see \cite[Theorem 1.1]{Jevtic-2017} and notice that on \(A^p_\alpha\), \(H_\lambda, \lambda>0\) is bounded if and only if \(H_1\) is bounded.

As with the Volterra operator, we have a representation of the mean of suitable weighted composition operators
\[
 H_\lambda(f) (z) = \int_0^1 u_t C_{\phi_t} \, dt,
\]
where \(u_t(z):=\frac{t^{\lambda -1}}{ ( 1-(1-t)z )^{\lambda} }, \phi_t(z) := \frac{t}{1-(1-t)z}, z\in D,\lambda >0\). Similarly to the case of the Volterra operator, it can be shown that \(\{ u_t C_{\phi_t} \colon t\in]0,1[\}\) is \hyperlink{admissible}{admissible}. Using Theorems \ref{thm:essentialNormEquality} and \ref{thm:DominatingAE}, we can now extend the result \cite[Example 7.4]{Lindstrom-2022} to
\[
\norm{H_{\lambda}}_{e, \BOP(A^p_\alpha)} = \int_0^1 \frac{ t^{(2+\alpha)/p}  }{ (1-t)^{(2+\alpha)/p} }    \, dt.
\]
Notice that using this approach, there is no need of finding a suitable family of compact operators, not only characterizing the \(M\)-ideal property, but also being equicontinuous w.r.t. \(\tau_0\)-topology, see \cite[condition (3.1) in Lemma 3.2]{Lindstrom-2022}. Notice that \cite[Proof of Lemma 3.2]{Lindstrom-2022} is unclear and the previous solution is rather involved, see \cite{Lindstrom-2024C}. The operator \(H_\lambda\) is also a special case of the generalized Hilbert matrix operator examined in \cite{Galanopoulos-2023}, see for example Theorem A (iii) for boundedness on weighted Bergman spaces.

\subsection{Another example}
Our final example is an integral operator of the form \(\int u_t C_{\phi_t} \, dt \) for which \(\phi_t\) is not univalent. First, we remark that given a function \(\phi\colon \D\to \C\), if  \(\phi(\D)\cup \{1\}\) contains a disk \(\subset \D\cup \{1\}\) touching \(1\), then so does \(\phi^2(\D)\cup \{1\}\). This can be seen as follows: Let \(f_{\theta}(z) = (1-\theta) + \theta z \) and notice that  \(f_{\theta_1}(\D) \subset f_{\theta_2}(\D)\) when \(0<\theta_1\leq \theta_2<1\). By assumption there exists a small \(0<\theta\leq \frac{1}{6}\) such that \(f_{\theta}(\D) \subset \phi(\D)\cup \{1\}\). It holds that 
\begin{equation}\label{eq:provingSquaredDiskContainsDisk}
\begin{split}
\inf &\{\abs{z} :  z\in \C\setminus ( f_{\theta}^2(\D) - (1-\theta) )   \} = \min_t \abs{ ( (1-\theta) + \theta e^{it})^2 - (1-\theta) } \\
%&=\theta \min_t \abs{  (1-\theta)(2e^{it} - 1 ) + \theta e^{2it}  } \\
&=\theta \min_t \sqrt{  ((1-\theta)(2\cos t - 1 ) + \theta \cos(2t))^2 + ( 2(1-\theta)\sin t  + \theta \sin(2t))^2  }.
\end{split}
\end{equation}
For \(t\in]0,\pi/4[\), we have \( 2\cos t - 1 \geq 2\cos^2 t - 1 =\cos (2t)  \) and \( 2\sin t  \geq 2\sin t \cos t  =\sin (2t)  \). For \(t\in [\pi/4, 3\pi/4]\), we have \(2(1-\theta)\sin t  -\theta \geq 1\), and for \(t\in [3\pi/4,\pi]\), we have \((1-\theta)\abs{2\cos t  - 1} -\theta \geq 1\). It follows from \eqref{eq:provingSquaredDiskContainsDisk} that
\[
\inf \{\abs{z} :  z\in \C\setminus ( f_{\theta}^2(\D) - (1-\theta) )   \} \geq \theta,
\]
which implies \(f_{\theta}(\D)-(1-\theta) \subset f_{\theta}^2(\D)-(1-\theta) \).

Let
\[
\phi_t(z) := \Big( \frac{  1+(1+t)z  }{  2+t  }  \Big)^2, \quad \phi_t'(z) =2(1+t) \frac{  1+(1+t)z  }{  (2+t)^2 } %, \quad   \phi_t^{-1}(z)= \frac{  (2+t)\sqrt{z}  - 1 }{  1+t  }.
\]
and for a continuous, positive function \(U\in L^1(]0,1[)\), define
\[
u_t (z) := U(t) \phi_t'(z)^{\frac{2}{p}}. 
\]
It is easy to see that the only nontrivial conditions for the family \(\{ u_tC_{\phi_t}:t\in]0,1[ \}\) to be \hyperlink{admissible}{admissible} are \ref{C:NeigborhoodContainsHorocylce} and \(\int \norm{u_tC_{\phi_t}}\, dt <\infty\). The remark above proves that  \ref{C:NeigborhoodContainsHorocylce}  is satisfied. Since \(\phi_t\) is univalent on \(\D_{\Im>0}:=\{ z\in \D : \Im z > 0  \} \) and \(\D_{\Im<0}\), and
\[
\frac{  1-\abs{z}^2  }{  1-\abs{ \phi_t(z)  }^2   }  \lesssim \frac{  1-\abs{z}  }{  1-\abs{ \frac{   1 + (1 + t)z    }{ 2+t }  }   } \lesssim 1,
\]
we have
\begin{align*}
 U(t)^{-p}\norm{u_tC_{\phi_t}f}_{A^p_{\alpha}}^p &= \int_{\D}  \abs{f(\phi_t(z))}^p \abs{\phi_t'(z)}^2 \, dA_\alpha(z) \\
&= \int_{\D_{\Im>0}}  \abs{f(\phi_t(z))}^p \abs{\phi_t'(z)}^2 \, dA_\alpha(z) + \int_{\D_{\Im<0}}  \abs{f(\phi_t(z))}^p \abs{\phi_t'(z)}^2 \, dA_\alpha(z) \\
& \lesssim \int_{  \phi_t(\D_{\Im>0}) }  \abs{f(z)}^p \, dA_\alpha(z) + \int_{\phi_t(\D_{\Im<0})}  \abs{f(z)}^p  \, dA_\alpha(z) \leq  2\norm{f}_{A^p_{\alpha}}^p.
\end{align*}

Since \( \{ u_t C_{\phi_t} : t\in]0,1[\}\) is \hyperlink{admissible}{admissible}, Theorems \ref{thm:essentialNormEquality} and \ref{thm:DominatingAE} yield
\[
\norm{ \int_0^1 u_t C_{\phi_t} \, dt }_{e,\BOP(A^p_{\alpha})} = \int_0^1  U(t) \Big( \frac{2+t}{2(1+t)} \Big)^{\frac{\alpha}{p}} \, dt.
\]

\section{Some further results concerning the essential norm}\label{sec:FurtherResults}

Following the proof of N. J. Kalton \cite[Theorem 2.4 \((1) \Rightarrow (2)\)]{Kalton-1993}, it is easy to see that if there is a \(\lambda>0\) such that \(\norm{L_n}\leq \lambda\) for every \(n\), then every \(L\in \CONV\{L_n, L_{n+1},\ldots\}\), also satisfies \(\norm{L}\leq \lambda\). Moreover, if we replace \(\max\{\norm{S},\norm{T}\}\) by \(\max\{\norm{S},\norm{T-K}\}\) for any compact operator \(K\), we can obtain the same contradiction at the end, since \(\psi(T) = \psi(T-K)\). Furthermore, concerning the extension of functionals \(\phi\in \COP(X)^*\) to \(\BOP(X)^*\), the assumption that \(X\) has a shrinking compact approximating sequence \((K_n)\) yields that for each  \(S,V\in \BOP(X)\),
\[
\lim_n \PAIR{V^*K_n^*S^* x^*, x^{**}} =\lim_n \PAIR{K_n^*S^* x^*, V^{**}x^{**}} = \PAIR{S^* x^*, V^{**}x^{**}} =  \PAIR{V^* S^* x^*, x^{**}}.
\]
As in \cite{Kalton-1993}, it follows that \(\phi(SV) := \lim_n \phi(SK_n V), \, S,V\in\BOP(X) \) gives a unique norm preserving extension. Since \((1),(2)\) and \((3)\) in \cite[Theorem 2.4]{Kalton-1993} are all equivalent, we have obtained:

\begin{prop}
If \(X\) is a separable Banach space, then \(\COP(X)\) is an M-ideal in \(\BOP(X)\) if and only if there exists a shrinking compact approximating sequence \((K_n)\subset B_{\COP(X)}\), such that for any \(S\in\COP,T,V\in \BOP(X)\) it holds that
\[
\lim_n\norm{S + T(I-K_n)V} \leq \max\{ \norm{S}, \norm{TV}_e \}.
\]
\end{prop}

\section{Further research}

Expanding  \eqref{eq:equalityOfEssentialNorm} to hold for families not necessarily consisting of weighted composition operators, \(\{S_t : t\in]0,1[ \}\), and possibly consider other spaces such as higher dimensional weighted Bergman spaces and \(A^p_\alpha\), when \(p=1\) or \(\alpha<0\). It would be interesting to find such families for which the essential norm cannot be calculated explicitely, but the geometric nature of the problem would ensure that  \eqref{eq:equalityOfEssentialNorm}  holds. Notice that a Banach space satisfying \((M_p)\) does not contain enough structure to be able to characterize equalities such as \ref{eq:equalityOfEssentialNorm}. The characterization obtained in this work heavily relies on the function space nature, in order to distinguish approximate evaluation maps from the general weakly null sequences. Moreover, the direction of the family is a concept that could be generalized, but also might demand some structure of the space.

\section{Acknowledgments}

The author was supported by the Magnus Ehrnrooth Foundation.

\bibliographystyle{abbrv}
\bibliography{bibliography}

\end{document}